\documentclass{amsart}
\usepackage{color, graphicx, enumerate, amssymb}
\usepackage{hyperref}
\usepackage{epsfig,wrapfig}
\usepackage{pxfonts}
\usepackage{graphicx}

\usepackage{eucal}

\usepackage[utf8]{inputenc}
\usepackage{ulem}

\usepackage{amsmath}

\usepackage{geometry}
\usepackage{enumitem}

\newtheorem{theorem}{Theorem}[section] 
\newtheorem{lemma}[theorem]{Lemma}
\newtheorem{corollary}[theorem]{Corollary}
\newtheorem{definition}{Definition}[section]
\newtheorem{proposition}[theorem]{Proposition}
\newtheorem{remark}[theorem]{Remark}

\newcommand{\dist}{{\rm dist}}
\newcommand{\cK}{{\mathcal K}}
\newcommand{\bR}{{\mathbb R}}

\def\avint{\mathop{\,\rlap{$\diagup$}\!\!\int}\nolimits}

\title[vectorial almost-minimizers]{Lipschitz regularity of a weakly coupled vectorial almost-minimizers for the $p$-Laplacian}

\author[M. Bayrami]{Masoud Bayrami}
\address{School of Mathematics, Institute for Research in Fundamental Sciences (IPM), P.O. Box: 19395-5746, Tehran, Iran}
\email{aminlouee@ipm.ir}

\author[M. Fotouhi]{Morteza Fotouhi}
\address{Department of Mathematical Sciences, Sharif University of Technology, P.O. Box: 11365-9415, Tehran, Iran}
\email{fotouhi@sharif.edu}

\author[H. Shahgholian]{Henrik Shahgholian}
\address{Department of Mathematics, KTH Royal Institute of Technology, 100 44, Stockholm, Sweden}
\email{henriksh@kth.se}

\date{\today}

\begin{document}

\begin{abstract}

For a given constant $\lambda > 0$ and a bounded Lipschitz domain $D \subset \mathbb{R}^n$ ($n \geq 2$), we establish that almost-minimizers of the functional
$$
 J(\mathbf{v}; D) = \int_D \sum_{i=1}^{m} \left|\nabla v_i(x) \right|^p+ \lambda \chi_{\{\left|\mathbf{v} \right|>0\}} (x) \, dx, \qquad 1<p<\infty,
 $$
where $\mathbf{v} = (v_1, \cdots, v_m)$, and $m \in \mathbb{N}$, exhibit optimal Lipschitz continuity in compact sets of $D$.
Furthermore, assuming $p \geq 2$ and employing a distinctly different methodology, we tackle the issue of boundary Lipschitz regularity for $v$. This approach simultaneously yields alternative proof for the optimal local Lipschitz regularity for the interior case.

\end{abstract}



\maketitle

\tableofcontents

\section{Introduction and the main results}

For $1<p<\infty$, $m\in \mathbb{N}$, a constant $\lambda>0$, and a bounded Lipschitz domain $D \subset \bR^n$ ($n \geq 2$), we will deal with  almost-minimizers of the  functional
\begin{equation}
\label{E0}
J(\mathbf{v}; D) = \int_D \sum_{i=1}^{m} \left|\nabla v_i(x) \right|^p + \lambda \chi_{\{ \left|\mathbf{v} \right|>0\}} (x) \, dx, 
\end{equation}
over  an admissible class 
$$ \cK:=\left\{ \mathbf{v} \in W^{1,p}(D; \mathbb{R}^m) \, : \, \text{$\mathbf{v}=\mathbf{g}$ on $\partial D$ and $v_i \geq 0$} \right\}. $$
Here $\mathbf{v}=(v_1, \cdots, v_m)$,  $|\mathbf{v}|=\sqrt{(v_1)^2+ \cdots + (v_m)^2}$,  and 
 $\mathbf{g}=(g_1,\cdots, g_m)$, $0\le g_i \in W^{1,p}(D)$.

\begin{definition}
We say that $\mathbf{u}=(u_1,\cdots,u_m)$ is a (local) almost-minimizer for $J$ in $D$, with constant $\kappa$ and exponent $\beta$, if 
$$
J(\mathbf{u}; B_r(x_0)) \leq \left(1+\kappa r^{\beta} \right) J(\mathbf{v}; B_r(x_0)), 
$$
for every ball $B_r(x_0)$ such that $\overline{B_r(x_0)} \subset D$ and every $\mathbf{v} \in W^{1,p}(B_r(x_0); \mathbb{R}^m)$ such that $\mathbf{u}=\mathbf{v}$ on $\partial B_r(x_0)$.
\end{definition}

It's worth noting that almost-minimizers are often referred to as $\omega$-minima in the literature, where $\omega: (0, r_0) \to [0,\infty)$, with $r_0>0$, represents a non-decreasing function with $\omega (0^+)=0$. In our context, within the given definition, this $\omega$ is represented as $\omega(r) = \kappa r^{\beta}$. When $\kappa=0$, the definition of an almost-minimizer reduces to the classical definition of a (local) minimizer for $J$.

\begin{definition}
By the free boundary of an almost-minimizer $\mathbf{u}$, we mean 
$$ F(\mathbf{u}):=\partial \{|\mathbf{u}| > 0\} \cap D, $$
where
$$ \{|\mathbf{u}| > 0\}=\cup_{i=1}^{m} \{u_i>0\}, $$
is called the positivity set of $\mathbf{u}$.
\end{definition}

Dealing with almost-minimizers presents a distinct challenge, as the conventional PDE framework isn't readily applicable for their further analysis. This differs markedly from minimizers, which conform to their corresponding Euler-Lagrange equation, neatly fitting within the PDE framework.

To elaborate, minimizers are associated with functional equations that squarely place them within the realm of PDEs—a tool not as readily accessible when addressing almost-minimizers. The concept of almost-minimizers isn't just a mathematical curiosity; it holds practical significance as well. It allows for the modeling of additional terms or disturbances that have a relatively minor impact at smaller scales.

In simpler terms, these almost-minimizers enable the characterization of perturbations with less explicit influence or disturbances stemming from noise. The versatility of this framework permits a broader scope of inquiries and facilitates the incorporation of slight errors and unpredictability into models.

To be more precise, in many scenarios, a minimizer for a complex functional often emerges as an almost-minimizer when dealing with a simplified/modified version of that functional, as exemplified in \cite{dipierro2022lipschitz}. For well-established findings concerning almost-minimizers of functionals with smooth integrands, we  refer to \cite[Chapter 7]{giusti2003direct} and \cite[Section 4.5]{mingione2006regularity}. Additionally, one can explore \cite[Appendix A]{de2022almost} for an instance of an almost-minimizer that serves as a solution to a singular system, inclusive of lower-order terms.

The local Lipschitz regularity of almost-minimizers in the context of Equation \eqref{E0} with specific parameters, namely $p=2$ and $m=1$, was initially established by David and Toro in \cite{MR3385167}. Subsequently, this result was extended to a certain range of values for the parameter $p$ by other researchers, as demonstrated in \cite{dipierro2022lipschitz} and \cite{forcillo2021regularity}. Additionally, the study of the semilinear case with variable coefficients is explored in \cite{MR3827045}.

Furthermore, for the case when $p=2$, the regularity theory pertaining to the free boundary has been investigated by David, Engelstein, and Toro in \cite{MR3948692}. This study was also conducted by De Silva and Savin in \cite{MR4126326} using the viscosity approach and improvements related to flatness. Moreover, an attempt has been initiated in the thesis \cite{forcillo2021regularity} to investigate this aspect for certain ranges of the parameter $p$ within the range of $1<p<\infty$.

Even when dealing with minimizers of Equation \eqref{E0}, there is a scarcity of results, particularly in the context of systems, as evidenced by references \cite{MR3827812}, \cite{MR4142361}, \cite{fotouhi-shahgholian2023}, and \cite{MR4085121}. However, our primary focus lies in studying almost-minimizers.

Our objective is to expand upon the findings presented in the recently published work \cite{dipierro2022lipschitz}, which addresses the local Lipschitz regularity of almost-minimizers in the context of Equation \eqref{E0}. Specifically, we aim to extend these results to cover any value of $p$ within the range of $1<p<\infty$ in a vectorial setting. It's worth noting that the outcomes reported in \cite{dipierro2022lipschitz} were limited to the parameter range where $p>\max \left\{\frac{2n}{n+2}, 1 \right\}$ and $m=1$.

Our approach for establishing the Lipschitz continuity of almost-minimizers will differ from the methodology employed in \cite{dipierro2022lipschitz}.

Our main result in this paper is presented in the following theorem.

\begin{theorem}[Local Lipschitz regularity of almost-minimizers]
\label{T1}
Assume that $ 1<p<\infty$, and let $\mathbf{u} :D \to \mathbb{R}^m $ be an almost-minimizer of $J$ in $D$. Then, $\mathbf{u}$ is locally Lipschitz continuous.
\end{theorem}

Additionally, in the last section, under the assumption $p \geq 2$, we explore the boundary Lipschitz continuity of almost-minimizers using a distinct method. This method, in turn, can be applied to achieve the same local result as outlined in Theorem \ref{T1}.

\medskip

The structure of the paper is outlined as follows:

In Section \ref{S2}, we demonstrate the local H\"older continuity of almost-minimizers. Subsequently, in Section \ref{n-section-3}, we show the local H\"older continuity of the gradient of almost-minimizers in the positivity set of $\mathbf{u}$.
Our reference for the proofs of the results in Sections \ref{S2} and \ref{n-section-3}, will be the work of De Filippis, \cite{de2020regularity}.
Moving on Section \ref{S2+}, we employ blow-up techniques to establish the local Lipschitz regularity of almost-minimizers, as outlined in Theorem \ref{T1}. Finally, Section \ref{LLL-section} is dedicated to study the boundary Lipschitz regularity, when also assuming $p\geq 2$.

\section{Partial regularity}
\label{M-S2}

In this section, our initial focus will be on establishing the $C^{0,\alpha}$-regularity of almost-minimizers, locally within $D$, for any exponent $\alpha \in (0,1)$.

\subsection{Local H\"older continuity of almost-minimizers}
\label{S2}

Before proceeding further, we introduce a useful lemma concerning the $p$-harmonic replacement of the components of the almost-minimizer $\mathbf{u}=(u_1, \cdots, u_m)$. We refer to ${u_i}^{*}_s$ as the $p$-harmonic replacement of ${u_i}$ in $B_s(x_0)$, which is the unique $p$-harmonic function in $B_s(x_0)$ with the same trace as ${u_i}$ on $\partial B_s(x_0)$.

\begin{lemma}
\label{Useful--Lemma}
Assume that $1<p<\infty$. Let $\mathbf{u}=(u_1, \cdots, u_m)$ be an almost-minimizer of $J$ in $D$, and let $x_0 \in D$ and $B_s(x_0) \Subset D$.
Define ${u_i}^{*}_s$ to be the $p$-harmonic replacement of ${u_i}$ in $B_s(x_0)$. Then,
\begin{equation}
\label{E4.5}
\int_{B_s(x_0)} \sum_{i=1}^{m} \left| \mathbf{V}(\nabla u_i) - \mathbf{V}(\nabla {u_i}^{*}_s) \right|^2 \, dx \leq c \kappa s^{\beta} \int_{B_{s}(x_0)} \sum_{i=1}^{m} \left|\nabla u_i \right|^p \, dx + c\lambda s^n,
\end{equation}
where the map $\mathbf{V}: \mathbb{R}^n \to \mathbb{R}^n$ is given by
\begin{equation}
\label{auxillary-map}
\mathbf{V}(z):= \left| z \right|^{\frac{p-2}{2}} z.
\end{equation}
\end{lemma}

\begin{proof}
Let $\mathbf{u}$ be an almost-minimizer of $J$ in $D$, and let $x_0 \in D$ and $B_s(x_0) \Subset D$. 
By the minimality of ${u_i}^{*}_s$'s, we know that
\begin{equation}
\label{Fi-3}
\int_{B_s(x_0)} \left|\nabla {u_i}^{*}_s \right|^p \, dx \leq \int_{B_s(x_0)} \left|\nabla u_i \right|^p \, dx,
\end{equation}
and also,
\begin{equation}
\label{E2-S-E}
\int_{B_s(x_0)} \left| \nabla {u_i}^{*}_s \right|^p   \,dx = \int_{B_s(x_0)} \left|\nabla {u_i}^{*}_s \right|^{p-2} \left( \nabla u_i \cdot \nabla {u_i}^{*}_s \right)  \, dx.
\end{equation}
On the other hand, since the following standard strict monotonicity inequality
\begin{equation}
\label{Fi-4}
c \left|\mathbf{V}(z_1)-\mathbf{V}(z_2) \right|^2 + p  |z_1|^{p-2} z_1 \cdot (z_2-z_1) \leq   |z_2|^p-|z_1|^p,
\end{equation}
holds true for the map \eqref{auxillary-map}, (see \cite[Section 3.3]{de2023nonuniformly}, especially (3.19) for a  proof for a general scenario);
hence, from \eqref{E2-S-E} and \eqref{Fi-4}, we obtain that 
$$
\begin{aligned}
\int_{B_s(x_0)} \sum_{i=1}^{m} \left|\mathbf{V}(\nabla u_i)-\mathbf{V}(\nabla {u_i}^{*}_s) \right|^2 \, dx \leq  c \int_{B_s(x_0)} \sum_{i=1}^{m} \left( \left|\nabla u_i \right|^p-\left|\nabla {u_i}^{*}_s \right|^p \right) \, dx.
\end{aligned}
$$
Finally, by using $\mathbf{v}=({u_1}^{*}_s, \cdots, {u_m}^{*}_s)$ in the definition of the almost-minimizer, $\mathbf{u}$, and also by \eqref{Fi-3}, we obtain  
$$ 
\begin{aligned}
 \int_{B_s(x_0)} \sum_{i=1}^{m} \left( \left|\nabla u_i \right|^p- \left|\nabla {u_i}^{*}_s \right|^p \right) \, dx 
 \leq \kappa s^{\beta} \int_{B_s(x_0)} \sum_{i=1}^{m} \left|\nabla {u_i}^{*}_s \right|^p \, dx + c \lambda s^n 
\leq \kappa s^{\beta} \int_{B_{s}(x_0)} \sum_{i=1}^{m} |\nabla u_i |^p \, dx + c \lambda s^n. 
\end{aligned}
$$
Therefore, 
we arrive  at \eqref{E4.5}, and the proof is complete.
\end{proof}

\begin{theorem}
\label{T3}
Assume that $1<p<\infty$. 
Let $\mathbf{u}=(u_1,\cdots,u_m)$ be an almost-minimizers of $J$ in $D$, with some positive constant $\kappa \le \kappa_0$ and exponent $\beta$. 
Then, $\mathbf{u}$ belongs to $C_{\mathrm{loc}}^{0,\alpha} \left(D; \mathbb{R}^m \right)$, for any $\alpha \in (0,1)$. More precisely, for any $\tilde{D} \Subset D$, there is constant $C=C(\tilde D,  p, n, \kappa_0, \beta)$ such that
$$ \|\mathbf{u}\|_{C^{0,\alpha}(\tilde{D}; \mathbb{R}^m)}\leq C\left(\|\nabla \mathbf{u}\|_{L^{p}(D; \mathbb{R}^m)}+\lambda^{\frac{1}{p}} \right). $$ 
\end{theorem}

\begin{proof}
Let $B_{r}(x_0) \Subset D$, $r \leq 1$.
Now, define $v_i:={u_i}^{*}_r$'s as the $p$-harmonic replacement of $u_i$'s in $B_r(x_0)$. 
For the $v_i$'s, we also have the following standard estimate (see \cite[Lemma 5.8]{diening2009everywhere})
\begin{equation}
\label{s-h-estimate-1}
\int_{B_t(x_0)}  |\nabla v_i|^p \, dx \leq c \left(\frac{t}{s} \right)^n \int_{B_s(x_0)} |\nabla v_i |^p \, dx,
\end{equation}
for a universal constant $c$ and for any $0<t<s \leq 1$. Now, we fix $\tau \in (0,1)$, recall the auxiliary map \eqref{auxillary-map}, and use $ |\mathbf{V}(z)|^2 = |z|^p $ and \eqref{s-h-estimate-1}, together with the estimate \eqref{E4.5} on the ball $B_r(x_0)$, to have
\begin{align} \nonumber
\int_{B_{ \tau r}(x_0)} \sum_{i=1}^{m} |\nabla u_i|^p \, dx & \leq c \left\{ \int_{B_{\tau r}(x_0)} \sum_{i=1}^{m} \left|\mathbf{V}(\nabla u_i) - \mathbf{V}(\nabla v_i) \right|^2 \, dx + \int_{B_{ \tau r}(x_0)} \sum_{i=1}^{m} |\nabla v_i|^p \, dx \right\} \\ \nonumber
& \leq c \left\{ \int_{B_r(x_0)} \sum_{i=1}^{m} \left|\mathbf{V}(\nabla u_i) - \mathbf{V}(\nabla v_i) \right|^2 \, dx + \tau^n \int_{B_r(x_0)} \sum_{i=1}^{m} |\nabla v_i|^p \, dx \right\} \\ \nonumber
& \leq c \left( \kappa r^{\beta}  + \tau^n \right) \int_{B_{r}(x_0)} \sum_{i=1}^{m} |\nabla u_i |^p \, dx + c \lambda r^n \\ \label{Fi-08}
& \leq \tau^{n-\epsilon} \left( r^{\beta} c \kappa_0  \tau^{ \epsilon -n}  + c \tau^{\epsilon} \right) \int_{B_{r}(x_0)} \sum_{i=1}^{m} |\nabla u_i |^p \, dx + c \lambda r^n ,
\end{align}
for any $\epsilon \in (0,n)$.
Next, fixing  $\tau \in (0,1)$ such that $c \tau^{\epsilon} < \frac{1}{2}$, and a threshold radius $0<r<R_* \leq 1$ such that $R_*^{\beta} c \kappa_0 \tau^{\epsilon-n}<\frac{1}{2}$, we have that  \eqref{Fi-08} becomes
$$
\int_{B_{\tau r}(x_0)} \sum_{i=1}^{m} |\nabla u_i|^p \, dx \leq \tau^{n-\epsilon} \int_{B_{r}(x_0)} \sum_{i=1}^{m} |\nabla u_i|^p \, dx + c \lambda r^n.
$$
In particular, for $k \in \mathbb{N}$, we obtain 
\begin{align} \nonumber
 \int_{B_{\tau^k r}(x_0)} \sum_{i=1}^{m} |\nabla u_i |^p \, dx& \leq \tau^{k(n-\epsilon)} \int_{B_{r}(x_0)} \sum_{i=1}^{m} |\nabla u_i|^p \, dx + c \lambda r^n  \frac{\tau^{k(n-\epsilon)}-\tau^{kn}}{\tau^{n-\epsilon}-\tau^{n}} \\ \label{Fi-11}
&\leq \tau^{k(n-\epsilon)} \int_{B_{r}(x_0)} \sum_{i=1}^{m} |\nabla u_i|^p \, dx + c \lambda \left(\tau^k r \right)^{n-\epsilon} .   
\end{align} 
When $0 < s < r \leq R_*$, it's straightforward to find a natural number $k$ such that 
$\tau^{k+1} r \leq s \leq \tau^k r$. Using equation \eqref{Fi-11}, we can then obtain:

\begin{align} \nonumber
\int_{B_s(x_0)} \sum_{i=1}^{m} |\nabla u_i|^p \, dx  & \leq \int_{B_{\tau^k r}(x_0)} \sum_{i=1}^{m} |\nabla u_i|^p \, dx \\ \nonumber
& \leq \tau^{\epsilon-n} \tau^{(k+1)(n-\epsilon)} \int_{B_{r}(x_0)} \sum_{i=1}^{m} |\nabla u_i|^p \, dx + c \lambda \tau^{\epsilon-n} \left(\tau^{k+1} r \right)^{n-\epsilon} \\ \label{Fi-12}
& \leq c \left( \frac{s}{r} \right)^{n-\epsilon} \int_{B_{r}(x_0)} \sum_{i=1}^{m} |\nabla u_i|^p \, dx + c \lambda s^{n-\epsilon}.
\end{align}
Upon relaxing the restriction $r \leq R_*$, we can distinguish between two scenarios: 
when $0 < s \leq R_* < r \leq 1$, and when 
$0 < R_* < s < r \leq 1$. In the first case, as indicated by \eqref{Fi-12}, we have:

$$
\begin{aligned}
\int_{B_s(x_0)} \sum_{i=1}^{m} |\nabla u_i|^p \, dx & \leq c \left( \frac{s}{R_*} \right)^{n-\epsilon} \int_{B_{R_*}(x_0)} \sum_{i=1}^{m} |\nabla u_i|^p \, dx + c \lambda s^{n-\epsilon} 
\\ & \leq c \left( \frac{s}{r} \right)^{n-\epsilon} \left( \frac{r}{R_*}\right)^{n-\epsilon} \int_{B_{r}(x_0)} \sum_{i=1}^{m} |\nabla u_i|^p \, dx + c \lambda s^{n-\epsilon}
\\& \leq \frac{c}{R_*^{n-\epsilon}} \left( \frac{s}{r} \right)^{n-\epsilon} \int_{B_{r}(x_0)} \sum_{i=1}^{m}
|\nabla u_i|^p \, dx +c \lambda s^{n-\epsilon}.
\end{aligned}
$$
While, in the second case, 
$$
\begin{aligned}
\int_{B_s(x_0)} \sum_{i=1}^{m} |\nabla u_i|^p \, dx & \leq  \left( \frac{s}{r} \right)^{n-\epsilon} \left( \frac{r}{R_*} \right)^{n-\epsilon} \int_{B_{s}(x_0)} \sum_{i=1}^{m} |\nabla u_i|^p \, dx + c \lambda s^{n-\epsilon}\\ 
& \leq  R_*^{\epsilon-n} \left( \frac{s}{r} \right)^{n-\epsilon} \int_{B_{r}(x_0)} \sum_{i=1}^{m} |\nabla u_i|^p \, dx + c \lambda s^{n-\epsilon} \\
& \leq c \left( \frac{s}{r} \right)^{n-\epsilon} \int_{B_{r}(x_0)} \sum_{i=1}^{m}
|\nabla u_i|^p \, dx + c \lambda s^{n-\epsilon}.
\end{aligned}
$$
Now,
\eqref{Fi-12} holds for any $0<s<r \leq 1$ and for a constant $c$ depends only on $p, n, \kappa_0$ and $\epsilon$. 

As a consequence of \eqref{Fi-12}, after a standard covering argument, for any fixed $\epsilon > 0$, we have 
\begin{equation}
\label{Fi-14}
\avint_{B_{r}(x_0)} \sum_{i=1}^{m} |\nabla u_i|^p \, dx \le  c \left(\|\nabla \mathbf{u}\|_{L^{p}(D; \mathbb{R}^m)}^p+\lambda \right) r^{-\epsilon},
\end{equation}
for any $0<r<1$, with a constant $c=c(p, n, \kappa_0, \epsilon)$.
From  Poincare's inequality and \eqref{Fi-14}, we see that
$$ 
\avint_{B_{r}(x_0)} \sum_{i=1}^{m} |u_i-{u_i}_{r}|^p \, dx \leq c r^p \avint_{B_{r}(x_0)} \sum_{i=1}^{m} |\nabla u_i|^p \, dx \leq C \left(\|\nabla \mathbf{u}\|_{L^{p}(D; \mathbb{R}^m)}^p +\lambda \right) r^{p-\epsilon},
$$
where ${u_i}_{r}$ denotes the average of $u_i$ in $B_{r}(x_0)$. Thus after covering, for any fixed $0<\epsilon <\min\{n,p\}$, 
by the Morrey and Campanato space embedding theorem, we have $ \mathbf{u} \in C_{\mathrm{loc}}^{0,1-\frac{\epsilon}{p}}(D; \mathbb{R}^m)$, hence $\mathbf{u}$ is locally H\"older continuous at any positive exponent less than one over $D$. 
\end{proof}

\begin{remark}\label{Remark-e-01}
 The given argument extends to $\partial D$ in the case of a uniformly $C^1$-smooth domain and Dirichlet data represented by $\mathbf{g}=(g_1, \cdots, g_m) \in C^{0,1}(D;\mathbb{R}^m)$. Specifically, one can replicate the steps outlined in the previous proof, but this time employing the subsequent boundary estimate for some $q>p$,
$$
\int_{B_t(x_0) \cap D}  |\nabla v_i|^p \, dx \leq c \left( \left(\frac{t}{s} \right)^{\tilde{n}}\int_{B_s(x_0) \cap D} |\nabla v_i |^p \, dx + t^{n\left(1-\frac{p}{q}\right)} \left(\int_{B_s(x_0) \cap D} |\nabla g_i|^q \, dx  \right)^{\frac{p}{q}} \right),
$$ 
for any $\tilde{n} \in \left[n\left(1-\frac{p}{q}\right),n\right) $, instead of \eqref{s-h-estimate-1} (see e.g. 
\cite[Lemma 3.4]{duzaar2004partial}, for the special case of flat boundary, $\partial D$, near $x_0$, and then by straightening out the 
boundary for the general case), to get the $C^{0,\alpha}$-regularity of almost-minimizers, up-to-the-boundary, for any exponent $\alpha \in (0,1)$. See \cite[Theorem 5.4]{duzaar2004partial} for more details.
Only notice that, the definition of almost-minimizer with the prescribed boundary value $\mathbf{g}$, should be revisited as follows: $\mathbf{u}$ is an almost-minimizer for $J$ in $D$, with  constant $\kappa$ and exponent $\beta$, if $\mathbf{u}-\mathbf{g} \in W_0^{1,p}(D; \mathbb{R}^m)$, and 
$$
J(\mathbf{u}; B_r(x_0) \cap D) \leq \left(1+\kappa r^{\beta} \right) J(\mathbf{v}; B_r(x_0) \cap D), 
$$
for every ball $B_r(x_0) \subset \mathbb{R}^n$ and every $\mathbf{v}$ such that $\mathbf{u}-\mathbf{v} \in W^{1,p}_0(B_r(x_0) \cap D; \mathbf{R}^m)$.
\end{remark}

\subsection{Local $C^{1,\eta}$-regularity of almost-minimizers in $\{|\mathbf{u}|>0\}$}
\label{n-section-3}

Now, in the following theorem, we focus on the local H\"older continuity of the gradient of almost-minimizers, away from the free boundary. While the proof of this result is well-established, at least for the case when $p=2$ and in the scalar scenario, as found in references such as \cite{anzellotti1983}, \cite[Section 3]{MR3385167}, or \cite[Chapter 8]{giusti2003direct}, we include the proof here for the sake of completeness.

\begin{theorem}
\label{T4}
Assume that $1<p<\infty$. Let $\mathbf{u}=(u_1,\cdots,u_m)$ be an almost-minimizers of $J$ in $D$, with some constant $\kappa\le \kappa_0$ and exponent $\beta$. 
Then, $\mathbf{u}$ belongs to $C_{\mathrm{loc}}^{1,\eta} \left(\{|\mathbf{u}| > 0\}; \mathbb{R}^m \right)$.
More precisely, for any $\tilde D\Subset \{|\mathbf{u}| > 0\}$, there is positive exponent $\eta=\eta(p, n, \beta)$ and constant $C=C(\tilde D,p, n, \kappa_0, \beta)$ such that 
\[
\|\mathbf{u}\|_{C^{1,\eta}(\tilde{D}; \mathbb{R}^m)}\leq C \left(
\|\nabla \mathbf{u}\|_{L^{p}(D; \mathbb{R}^m)}
+\lambda^{\frac{1}{p}} \right). 
\]
\end{theorem}

\begin{proof}
First, we claim that 
$\mathbf{V}(\nabla u_i)$ is locally H\"older continuous in $\{|\mathbf{u}|>0\}$, where $\mathbf{V}$ is the map defined in \eqref{auxillary-map}.

Let $B_{r}(x_0) \Subset \{|\mathbf{u}|>0\}$, $r \leq 1$.
Similar to the proof of the estimate \eqref{E4.5} in Proposition \ref{Useful--Lemma}, define the $v_i:={u_i}^{*}_r$ as the $p$-harmonic replacement of $u_i$ in $B_r(x_0)$.
By using $\mathbf{v}=(v_1, \cdots, v_m)$ in the definition of the almost-minimizer, $\mathbf{u}$, and noting that by the maximum principle $B_{r}(x_0) \Subset \{|\mathbf{v}|>0\}$ also, we can easily obtain 
\begin{equation}
\label{Fi-05-t2}
\int_{B_r(x_0)} \sum_{i=1}^{m} \left|\mathbf{V}(\nabla u_i)-\mathbf{V}(\nabla v_i) \right|^2 \, dx \leq c \kappa r^{\beta} \left( 
\int_{B_{r}(x_0)} \sum_{i=1}^{m} \left|\nabla u_i \right|^p \, dx
+ \lambda r^{n} \right).
\end{equation}

By following the proof of inequality \eqref{Fi-14}, which commenced with the use of estimate \eqref{E4.5}; this time, invoking \eqref{Fi-05-t2} in place of \eqref{E4.5}, and repeating the same arguments as in \eqref{Fi-08}, \eqref{Fi-11}, and \eqref{Fi-12}, results in 

\begin{equation}
\label{Fi-14-new}
\avint_{B_{r}(x_0)} \sum_{i=1}^{m} |\nabla u_i|^p \, dx < c \left( \|\nabla \mathbf{u}\|_{L^{p}(D; \mathbb{R}^m)}^p + \lambda \right) r^{-\epsilon},
\end{equation}
for any fixed $\epsilon > 0$, while $c$ depending on $n,p, \kappa_0$, $\beta$ and $\epsilon$; we fix $\epsilon>0$ such that $\zeta:=\beta-\epsilon>0$. 
Moreover, we recall from \cite[Theorem 6.4]{diening2009everywhere}, the following well-known estimate 
\begin{equation}
\label{S-W-eq-0}
\avint_{B_s(x_0)} \left|\mathbf{V}(\nabla v_i)-\left(\mathbf{V}(\nabla v_i) \right)_s \right|^2 \, dx \leq c \left(\frac{s}{t} \right)^{\tilde{\mu}} \avint_{B_t(x_0)}  |\nabla v_i|^p \, dx,
\end{equation}
for any $0<s<t\leq r$, where $\left(\mathbf{V}(\nabla v_i) \right)_s$ is the average of $\mathbf{V}(\nabla v_i)$ in $B_s(x_0)$, and $\tilde{\mu}=\tilde{\mu}(p, n)$ a positive universal constant. Using \eqref{Fi-05-t2}, \eqref{Fi-14-new} and \eqref{S-W-eq-0}, for $0<s<r$, we have 
\begin{align} \nonumber
 \avint_{B_s(x_0)} \sum_{i=1}^{m} \left|\mathbf{V}(\nabla u_i)-\left(\mathbf{V}(\nabla u_i) \right)_s \right|^2 \, dx &\leq c \Bigg(  \left( \frac{r}{s} \right)^n \avint_{B_{r}(x_0)} \sum_{i=1}^{m} \left|\mathbf{V}(\nabla u_i)-\mathbf{V}(\nabla v_i) \right|^2 \, dx \\ \nonumber & 
\quad \qquad \,\,\, + \avint_{B_s(x_0)} \sum_{i=1}^{m} \left|\mathbf{V}(\nabla v_i)- \left(\mathbf{V}(\nabla v_i) \right)_s \right|^2 \, dx \\ \nonumber
& \quad \qquad \,\,\,  + \avint_{B_s(x_0)} \sum_{i=1}^{m} \left| \left(\mathbf{V}(\nabla v_i) \right)_s -\left(\mathbf{V}(\nabla u_i) \right)_s \right|^2 \, dx \Bigg) \\ \label{Fi-H-1}
&  \leq c \left( \|\nabla \mathbf{u}\|_{L^{p}(D; \mathbb{R}^m)}^p + \lambda \right) \left( \left( \frac{r}{s} \right)^n r^{\zeta} + \left( \frac{s}{r}\right)^{\tilde{\mu}} r^{-\epsilon} 
\right),
\end{align}
where for the last term we have used the following estimate 
$$ 
\begin{aligned}
\sum_{i=1}^{m} & \left| \left(\mathbf{V}(\nabla v_i) \right)_s -\left(\mathbf{V}(\nabla u_i) \right)_s \right|^2  
= \sum_{i=1}^{m} \left|\avint_{B_s(x_0)} \left( \mathbf{V}(\nabla v_i) - \mathbf{V}(\nabla u_i) \right) \, dx \right|^2 \\
 & \quad \le \sum_{i=1}^{m} \avint_{B_s(x_0)} \left|\mathbf{V}(\nabla v_i) - \mathbf{V}(\nabla u_i) \right|^2 dx 
\leq \left( \frac{r}{s} \right)^n \avint_{B_{r}(x_0)} \sum_{i=1}^{m} \left|\mathbf{V}(\nabla u_i)-\mathbf{V}(\nabla v_i) \right|^2 \, dx. 
\end{aligned}
$$ 
Now by setting $s:=r^{1+a}$ in \eqref{Fi-H-1}, and  equalizing 
the right-hand side \eqref{Fi-H-1} with 
$a=\frac{\zeta+\epsilon}{\tilde{\mu}+n}= \frac{\beta}{\tilde{\mu}+n}$; and further selecting $\epsilon$ small enough such that $\epsilon < a\tilde \mu$,
after standard computations, we end up with
$$
\avint_{B_s(x_0)} \sum_{i=1}^{m} \left| \mathbf{V}(\nabla u_i)-\left(\mathbf{V}(\nabla u_i) \right)_s \right|^2 \, dx \leq c \left( \|\nabla \mathbf{u}\|_{L^{p}(D; \mathbb{R}^m)}^p + \lambda \right) s^{\theta_0},
$$ 
for some positive constant $\theta_0$. 
Now, Morrey and Campanato space embedding theorem implies $C^{0,\theta_0}$-regularity of $\mathbf{V}(\nabla u_i)= |\nabla u_i|^{\frac{p-2}2}\nabla u_i$, locally in $\{|\mathbf{u}|>0\}$.
Hence, we can conclude that $\nabla u_i \in C_{\mathrm{loc}}^{0,\eta}(\{|\mathbf{u}|>0\})$ for some $\eta\in(0,1)$.
\end{proof}

\begin{remark}
\label{Remark-e-02}
Again, for $C^{1,\alpha}$-smooth domain $D$, for some $\alpha>0$, and 
 Dirichlet data $\mathbf{g}=(g_1, \cdots, g_m) \in C^{1,\alpha}(D;\mathbb{R}^m)$,
by using the following boundary estimate (see e.g. (4.34) in \cite[Lemma 4.5]{beck2008boundary} for the special case of $\mathbf{g}=\mathbf{0}$, and merging the results of Lemmas 4 and 5 in \cite{lieberman1988boundary} for the case of non-homogeneous data, both for the special case of flat boundary, $\partial D$, near $x_0$, and then by straightening out the 
boundary for the general case) 
$$
\avint_{B_s(x_0) \cap D} \left|\mathbf{V}(\nabla v_i)-\left(\mathbf{V}(\nabla v_i) \right)_s \right|^2 \, dx \leq c \left(\frac{s}{t} \right)^{\tilde{\mu}} \left(\avint_{B_t(x_0) \cap D}  |\nabla v_i|^p \, dx + \| g_i\|_{C^{1,\alpha}(D)}^p\right),
$$
instead of \eqref{S-W-eq-0}, one can prove the $C^{1,\tilde{\eta}}$-regularity of almost-minimizers up-to that boundary points of $\partial D$  which does not belong to $\overline{F(\mathbf{u})}$, for some $\tilde{\eta} \in (0,1)$.
\end{remark}

\begin{remark}
The results presented in this section, namely Theorems \ref{T3} and \ref{T4}, can evidently be extended to two-phase functionals, i.e.,  without imposing the non-negativity hypothesis on the  components of vectors in the  admissible class.
\end{remark}

\section{Lipschitz regularity of  local almost-minimizers}
\label{S2+}

To establish the Lipschitz continuity of almost-minimal solutions of $J$, we will employ the subsequent proposition, which is an inequality of Caccioppoli type for almost-minimal solutions of $J$.

\begin{proposition}
\label{almost-minimizer-caccioppoli}
Let $1<p<\infty$, and assume that $\mathbf{u}=(u_1, \cdots,u_m)$ is an almost-minimizer of $J$ in $D$, with some constant $\kappa\le \kappa_0$ and exponent $\beta$. 
Then, 
for any $B_r(x_0) \Subset D$, we have  
\begin{equation}
\label{G-04.1}
\int_{B_{\frac{r}{2}}(x_0)} \sum_{i=1}^{m} |\nabla u_i|^p \, dx \leq c \sum_{i=1}^{m} \left( \frac{1}{r^p} \int_{B_r(x_0)} |u_i-{u_i}_r|^p \, dx + \lambda r^n \right),
\end{equation}
for some universal constant $c=c(p, m, r, \kappa_0, \beta)$, where ${u_i}_r$ indicates the average of $u_i$ in $B_{r}(x_0)$
\end{proposition}

\begin{proof}
Let $B_r(x_0)$ be a sphere strictly contained in $D$, and let $\frac{r}{2} < s< t \leq r$. Let $\eta(x)$ be a function in $C_c^{\infty}(B_t(x_0))$, with $0 \leq \eta \leq 1$, $\eta \equiv 1$ in $B_s(x_0)$, and $|\nabla \eta| \leq \frac{2}{t-s}$. Denoting ${u_i}_t$ as the average of $u_i$ in $B_{t}(x_0)$, we set $\varphi_i=\eta(u_i-{u_i}_t)$, and consider
\begin{equation}
\label{G-04.3}
\int_{B_t(x_0)} | \nabla \varphi_i|^p \, dx = \int_{B_t(x_0)} | \nabla u_i|^p + \lambda \chi_{\{u_i>0\}} \, dx + \int_{B_t(x_0)}  \left( | \nabla \varphi_i|^p - | \nabla u_i|^p - \lambda \chi_{\{u_i>0\}}\right) \, dx.
\end{equation}
Now, let $v_i=u_i-\varphi_i={u_i}_t+(1-\eta)(u_i-{u_i}_t)$. 
By the definition of almost-minimizer, $\mathbf{u}$, we have 
$$
\begin{aligned}
\int_{B_t(x_0)} \sum_{i=1}^m|\nabla u_i|^p + \lambda \chi_{\{|\mathbf{u}|>0\}} \, dx & \leq \left(1+\kappa t^{\beta} \right) \int_{B_t(x_0)} \sum_{i=1}^m|\nabla v_i|^p + \lambda \chi_{\{|\mathbf{v}|>0\}} \, dx \\
& \leq \left(1+\kappa t^{\beta} \right) \int_{B_t(x_0)} \sum_{i=1}^m|\nabla v_i|^p + \lambda \, dx.
\end{aligned}
$$
We remark now that $\nabla \varphi_i = \nabla u_i$ in $B_s(x_0)$, and therefore the second integral on the right-hand side of \eqref{G-04.3} can be estimated by
$$ \int_{B_t(x_0) \setminus B_s(x_0)} \left(  | \nabla u_i|^p + | \nabla v_i|^p + \lambda \right) \, dx. $$
Introducing these relations in \eqref{G-04.3}, we get 
$$ \int_{B_s(x_0)} \sum_{i=1}^m|\nabla u_i|^p \, dx 
\leq \int_{B_t(x_0) \setminus B_s(x_0)}\sum_{i=1}^m |\nabla u_i|^p \, dx + c \int_{B_t(x_0)} \sum_{i=1}^m|\nabla v_i|^p \, dx + c \lambda \left| B_t(x_0) \right|. $$
We have
$$
\begin{aligned}
|\nabla v_i|^p &= \left|(1-\eta) \nabla u_i - (u_i-{u_i}_t) \nabla \eta \right|^p \\
& \leq c \left( (1-\eta)^p |\nabla u_i|^p + (t-s)^{-p} |u_i-{u_i}_t|^p \right).
\end{aligned}
$$
Thus, we obtain
$$
\begin{aligned}
\int_{B_s(x_0)} & \sum_{i=1}^m|\nabla u_i|^p \, dx 
\leq c\sum_{i=1}^m \left( \int_{B_t(x_0) \setminus B_s(x_0)} |\nabla u_i|^p \, dx + \frac{1}{(t-s)^p} \int_{B_t(x_0)} |u_i-{u_i}_t|^p \, dx + \lambda \left| B_r(x_0) \right| \right).
\end{aligned}
$$
Moreover, since we have 
$$ 
\begin{aligned}
\int_{B_t(x_0)} |u_i-{u_i}_t|^p \, dx 
& \leq 2^{p-1} \int_{B_t(x_0)} |u_i-{u_i}_r|^p + |{u_i}_r-{u_i}_t|^p \, dx \\
& \leq 2^{p-1} \int_{B_t(x_0)} |u_i-{u_i}_r|^p \, dx + 2^{p-1} |B_t(x_0)| \left| \avint_{B_t(x_0)} (u_i-{u_i}_r) \, dx \right|^p \\
& \leq 2^{p-1} \int_{B_t(x_0)} |u_i-{u_i}_r|^p \, dx + 2^{p-1} |B_t(x_0)|^{1-p} \int_{B_t(x_0)} |u_i-{u_i}_r|^p \, dx \left( \int_{B_t(x_0)} 1 \, dx \right)^{p-1} \\
& \leq c \int_{B_r(x_0)} |u_i-{u_i}_r|^p \, dx,
\end{aligned}
$$
we get
$$
\begin{aligned}
\int_{B_s(x_0)} & \sum_{i=1}^m |\nabla u_i|^p \, dx  \leq c \sum_{i=1}^m\left( \int_{B_t(x_0) \setminus B_s(x_0)} |\nabla u_i|^p \, dx + \frac{1}{(t-s)^p} \int_{B_r(x_0)} |u_i-{u_i}_r|^p \, dx + \lambda \left| B_r(x_0) \right| \right).
\end{aligned}
$$
Now, by the Widman's hole-filling argument (\cite{widman1971holder}), i.e. adding $c \sum_{i=1}^{m}\int_{B_s(x_0)}  |\nabla u_i|^p  \, dx$ on both sides of the preceding inequality, and subsequently dividing by $c+1$, we obtain
$$
\begin{aligned}
\int_{B_s(x_0)} & \sum_{i=1}^m|\nabla u_i|^p \, dx \leq \frac{c}{c+1} \int_{B_t(x_0)} \sum_{i=1}^m|\nabla u_i|^p \, dx + \frac{c}{(t-s)^p} \sum_{i=1}^m\int_{B_r(x_0)} |u_i-{u_i}_r|^p \, dx + c\lambda \left| B_r(x_0) \right|.
\end{aligned}
$$
Applying \cite[Lemma 6.1]{giusti2003direct} with
$$ I_0(s)=\int_{B_s(x_0)} \sum_{i=1}^m|\nabla u_i|^p \, dx, $$
and
$$ A=\sum_{i=1}^m\int_{B_r(x_0)} |u_i-{u_i}_r|^p \, dx, \qquad B=0, \qquad C=c\lambda \left| B_r(x_0) \right|, $$
we obtain immediately the inequality \eqref{G-04.1}.
\end{proof}

For the remainder of the paper, we will denote with $B_r:=B_r(\mathbf{0})$.

\begin{corollary}
\label{imp-cor}
Let $1<p<\infty$, and assume that $\mathbf{u}=(u_1, \cdots,u_m)$ is an almost-minimizer of $J$ in $B_1$ with some constant $\kappa\le \kappa_0$ and exponent $\beta$. Moreover, assume that $B_1=\{|\mathbf{u}| > 0\}$. Then 
\[
|\nabla \mathbf{u}(\mathbf{0})| \le C \left(
\| \mathbf{u}\|_{L^\infty(B_1; \mathbb{R}^m)} + \lambda^{\frac{1}{p}} \right),
\]
for a universal constant $C= C(p, m, n, \kappa_0, \beta)$.
\end{corollary}

\begin{proof}
By the proof of Theorem \ref{T4}, we have 
$$ \|\mathbf{u} \|_{C^{1,\eta} (B_{\frac{1}{4}}; \mathbb{R}^m )} \leq C \left( 
\|\nabla \mathbf{u}\|_{L^{p} (B_{\frac{1}{2}}; \mathbb{R}^m )} + \lambda^{\frac{1}{p}} \right). 
$$
On the other hand, by invoking Proposition \ref{almost-minimizer-caccioppoli}, we obtain 
$$ \|\mathbf{u} \|_{C^{1,\eta} (B_{\frac{1}{4}}; \mathbb{R}^m )} \leq C \left( \|\mathbf{u}\|_{L^{\infty}(B_1; \mathbb{R}^m)} + \lambda^{\frac{1}{p}} \right), $$
which completes the 
 proof of corollary.
\end{proof}

In the next proposition, we will use the following notation
$$ \mathbf{u}_{r,T} (x):= \frac{\mathbf{u}(rx)}{T}. $$

\begin{proposition}
\label{LConverg}
Assume that $1<p<\infty$. Let $\mathbf{u}^j= \left(u^j_1, \cdots, u^j_m \right)$ be a sequence of bounded almost-minimizers of $J$ in $B_2$. 
Also, set 
$$
\mathbf{v}^j(x):=\mathbf{u}^j_{r_j, T_j} (x)= \frac{\mathbf{u}^j(r_jx)}{T_j}, \qquad \text{in} \quad B_{2R},
$$
with $0<R<\frac{1}{r_j}$, where $r_j \to 0$, as $j \to \infty$, and $T_j >0$. 
Then, $\mathbf{v}^j= \left(v_1^j, \cdots,v_m^j \right)$ is the almost-minimizer (according to its own boundary values) of the following scaled functional
$$ \hat{J}_j \left(\mathbf{v}^j; D \right):=\int_{D} 
 \sum_{i=1}^{m} \left|\nabla v_i^j \right|^p + \sigma_j^p  \lambda \chi_{\left\{\left|\mathbf{v}^j \right|>0 \right\}} \, dx,
$$
with the constant $\hat{\kappa}= \kappa r_j^{\beta} $ and exponent $\hat{\beta}=\beta$, where  $\sigma_j:=\frac{r_j}{T_j}$.
Moreover, if $\left|\mathbf{v}^j \right| \leq M$ in $B_{2R}$, for any fixed $0<R<\frac{1}{r_j}$, and for some $M=M(R)>0$, then up to a subsequence, the following holds  
\begin{enumerate}
\item [(i)]
$v_i^j \to v_i^{\infty}$ weakly in $W^{1,p}(B_R)$, and also in $C^{0,\alpha}(B_R)$, for any $\alpha<1$.
\item [(ii)]
Besides, if $\sigma_j =\frac{r_j}{T_j} \to 0$, as $j \to  \infty$, then $v_i^{\infty}$ is a $p$-harmonic function in $B_R$.
\end{enumerate}
\end{proposition}

\begin{proof}[Proof of Proposition \ref{LConverg}] 
First of all, we show that $\mathbf{v}^j$ is an almost-minimizer of $\hat{J}_j$, with the constant $\hat{\kappa}=\kappa r_j^{\beta}$ and exponent $\hat{\beta}=\beta$; namely
\begin{equation}
\label{Ext-pro-00}
\hat{J}_j \left(\mathbf{v}^j; B_{\rho}(x_0) \right) \leq \left(1+\kappa r_j^{\beta} \rho^{\beta} \right) \hat{J}_j(\mathbf{w}; B_{\rho}(x_0)), 
\end{equation}
for every ball $ B_{\rho}(x_0)$ such that $ \overline{B_{\rho}(x_0)} \subset B_{\frac{1}{r_j}}$, and for every $\mathbf{w} \in W^{1,p}(B_{\rho}(x_0); \mathbb{R}^m)$ such that $\mathbf{w} = \mathbf{v}^j$ on $\partial B_{\rho}(x_0)$.

By almost-minimality,   we have 
$$ J \left(\mathbf{u}^j; B_{r_j\rho}(y_0) \right) \leq \left(1+\kappa (r_j\rho)^{\beta} \right) J \left(\mathbf{w}^j; B_{r_j\rho}(y_0)\right), 
$$
where $y_0=r_j x_0$ and $\mathbf{w}^j(x)=T_j \mathbf{w} (\frac{x}{r_j} )$. Thus
\begin{align} \label{Ext-pro-2}
\int_{B_{r_j\rho}(y_0)} &\sum_{i=1}^{m} \left|\nabla u_i^j(y) \right|^p + \lambda \chi_{\left\{\left|\mathbf{u}^j \right|>0 \right\}} \, dy  
 \leq \left(1+\kappa r_j^{\beta} \rho^{\beta}\right) \int_{B_{r_j \rho}(y_0)} \sum_{i=1}^{m} \left|\nabla w^j_i(y) \right|^p + \lambda \chi_{\left\{ \left|\mathbf{w}^j \right|>0 \right\}} \, dy.
\end{align}
Furthermore, using
the change of variable $y := {r_j}x$ in the right-hand side \eqref{Ext-pro-2}, we see that
\begin{align} \nonumber
\int_{B_{r_j\rho}(y_0)} \sum_{i=1}^{m} &  \left|\nabla w^j_i(y) \right|^p + \lambda \chi_{\left\{\left|\mathbf{w}^j \right|>0 \right\}} \, dy = r_j^n \int_{B_{\rho}(x_0)} \sum_{j=1}^{m} \left|\nabla w_i^j(r_jx) \right|^p + \lambda \chi_{\left\{ \left|\mathbf{w}^j \right|>0 \right\}}(r_jx) \, dx \\ \label{Ext-pro-3}
& = r_j^n \int_{B_{\rho}(x_0)} \left(  \frac{T_j}{r_j} \right)^p \sum_{i=1}^{m} \left|\nabla {w_i}(x) \right|^p + \lambda \chi_{\left\{\left|{\mathbf{w}}\right|>0\right\} }(x) \, dx,
\end{align}
and a similar identity holds true with $\mathbf{u}^j$ and $\mathbf{v}^j={\mathbf{u}}^j_{r_j,T_j}$
replacing $\mathbf{w}^j$ and ${\mathbf{w}}$. 
Plugging this information and \eqref{Ext-pro-3} into \eqref{Ext-pro-2}, we obtain 
$$ \int_{B_{\rho}(x_0)} \sum_{i=1}^{m} \left|\nabla v_i^j \right|^p + \left(  \frac{r_j}{T_j} \right)^p\lambda \chi_{\left\{\left|\mathbf{v}^j \right|>0 \right\}} \, dx \leq \left(1+\kappa r_j^{\beta} \rho^{\beta} \right) \int_{B_{\rho}(x_0)} \sum_{i=1}^{m} |\nabla w_i|^p + \left(  \frac{r_j}{T_j} \right)^p \lambda \chi_{\left\{|\mathbf{w}|>0 \right\}} \, dx, $$
i.e. the desired result \eqref{Ext-pro-00}.

Moreover, using $\left|\mathbf{v}^j \right|\leq M$ in $B_{2R}$, and 
by the help of Proposition \ref{almost-minimizer-caccioppoli}, one can obtain the uniform $W^{1,p}(B_R)$ estimates for $v_i^j$, $j$ large enough. Thus, at least for a subsequence, we get
$$ v_i^j \to v_i^{\infty} \qquad \text{weakly in} \quad W^{1,p}(B_R). $$
Moreover, applying the uniform H\"older estimates of $v_i^j$, Theorem \ref{T3}, we get
$$ v_i^j \to v_i^{\infty} \qquad \text{in} \quad C^{0,\alpha}(B_R),$$
for any $\alpha<1$. This completes the proof of \textup{(i)}.

Now, define $z_i^j$ as the $p$-harmonic replacement of $v_i^j$ in $B_{R}$. Then, Lemma \ref{Useful--Lemma} implies that
\begin{equation}
\label{E4.5-n}
\int_{B_{R}} \left| \mathbf{V} \left(\nabla v_i^j \right) - \mathbf{V} \left(\nabla z_i^j \right) \right|^2 \, dx \leq c \kappa r_j^{\beta} R^{\beta} \int_{B_{R}} \sum_{i=1}^{m} \left|\nabla v_i^j \right|^p \, dx + c \sigma_j^p \lambda R^n,
\end{equation}
Now, passing to the limit $j \to \infty$ in \eqref{E4.5-n}, we obtain that, up to a subsequence, $\mathbf{V} \left(\nabla v_i^j \right) - \mathbf{V} \left(\nabla z_i^j \right) \to \mathbf{0}$,  in $L^{2}(B_R; \mathbb{R}^n)$. Now, it follows that $v_i^j-z_i^j \to 0$, in $W^{1,p}(B_R)$, and
hence, $z_i^j\to v_i^{\infty}$ weakly in $W^{1,p}(B_R)$. 
On the other hand, since $z_i^j$'s are $p$-harmonic in $B_R$, thus $v_i^{\infty}$ should be $p$-harmonic in $B_R$, too. This finishes the proof of \textup{(ii)}. 
\end{proof}

Finally, we are prepared to demonstrate the Lipschitz continuity of almost-minimizers. 
Before that, we have the following proposition regarding the linear growth of almost-minimizers at the free boundary points.

\begin{proposition}
\label{SSSP-01}
Assume that $1<p<\infty$. Let $\mathbf{u}=(u_1, \cdots, u_m)$ be an almost-minimizer of $J$ in $B_1(x_0)$, and $x_0 \in F(\mathbf{u})$. 
Also, suppose that $\sup_{B_1(x_0)} |\mathbf{u}| \leq M$. 
Then, there exists a universal constant $C\ge 1$ depending on $M$, such that
$$ 0 \leq |\mathbf{u}(x)| \leq CM |x-x_0|, $$
for all $x \in B_{r}(x_0)$, and any $0<r<1$.
\end{proposition}

\begin{proof}
First notice that, without loss of generality, we may assume that $x_0 = 0$. Also, it is enough to prove that there exists a universal constant $\tilde C\ge 1$ such that
\begin{equation}
\label{SSS1}
S(k+1, \mathbf{u}) \leq \max \left\{ \frac{\tilde CM}{2^{k+1}}, \frac{S(k,\mathbf{u})}{2} \right\},
\end{equation}
where $ S(k,\mathbf{u}) := \sup_{B_{2^{-k}}} |\mathbf{u}|$. 
The reason is that, we can deduce inductively from \eqref{SSS1} that
$$ S(k,\mathbf{u}) \leq  \tilde CM 2^{-k}. $$
Now,  for an arbitrary $r \in (0,1]$, choose $k \geq 0$ such that $2^{-(k+1)}< r \le 2^{-k}$. Then,
$$
\|\mathbf{u}\|_{L^\infty(B_{r})} \leq \|\mathbf{u}\|_{L^\infty \left(B_{2^{-k}} \right)} = S(k,\mathbf{u}) \leq \tilde{C}M 2^{-k} \leq \left(2 \tilde{C} \right)M 2^{-(k+1)} \leq \left(2 \tilde{C} \right)M r,
$$
which provide the proof of the proposition.

 Hence, to prove \eqref{SSS1}, let's assume the contrary, and towards a contradiction, suppose otherwise. Then, there exist almost-minimizers $\mathbf{u}^j$ of $J$ in $B_1$, and integers $k^j$, $j=1,2, \cdots$, such that:

\begin{equation}
\label{SSS2}
S \left(k^j+1,\mathbf{u}^j \right) > \max \left\{ \frac{jM}{2^{k^j+1}}, \frac{S \left(k^j,\mathbf{u}^j \right)}{2} \right\}.
\end{equation} 
Here $ S \left(k^j,\mathbf{u}^j \right) :=\sup_{B_{2^{-k^j}}} \left|\mathbf{u}^j \right|$, and $ \left|\mathbf{u}^j \right|  \leq M$.
Observe that $ \left|\mathbf{u}^j \right| \leq M$ implies $k^j \to \infty$. 
Now, define the following scaled auxiliary function
$$ \mathbf{v}^j(x) := \frac{\mathbf{u}^j \left(2^{-k^j}x \right)}{S \left(k^j+1,\mathbf{u}^j \right)}, \qquad \text{in} \quad B_{2^{k^j}},$$
and let
$$ \sigma_j := \frac{2^{-k^j}}{S\left(k^j+1,\mathbf{u}^j\right)}, $$
which $\sigma_j \leq 2j^{-1}M^{-1} \to 0$, by \eqref{SSS2}. 
Also, consider the following scaled energy functional
$$
\hat{J}_j \left(\mathbf{v}^j; D \right):=\int_{D} \sum_{i=1}^{m} \left|\nabla v_i^j \right|^p + \sigma_j^p \lambda \chi_{\left\{\left|\mathbf{v}^j \right|>0 \right\}} \, dx.
$$
Since $\left|\mathbf{v}^j \right|\leq 2$ in $B_{1}$, due to \eqref{SSS2}, 
so by Proposition \ref{LConverg}, $\mathbf{v}^j=\left(v_1^j, \cdots, v_m^j \right)$ is an almost-minimizer of $\hat{J}_j$ with the constant $\hat{\kappa}=\kappa r_j^{\beta}$ and exponent $\hat{\beta}=\beta$, and $v_i^{\infty}:=\lim_{j \to \infty} v_i^j$ satisfies $\Delta_p v_i^{\infty} = 0 $ in $B_{\frac12}$. On the other hand, since 
\begin{itemize}
\item
$0 \leq v_i^{\infty} \leq 2$ in $B_{\frac12}$; 
\item
$v_i^{\infty}(\mathbf{0})=0$;
\item
$\sup_{B_{\frac{1}{2}}} v_i^{\infty} =1$; 
\end{itemize}
we arrive at   a contradiction with the strong minimum principle. 
\end{proof}

\begin{proof}[Proof of Theorem \normalfont{\ref{T1}}]
Let $\mathbf{u}$ be an almost-minimizer of $J$ in $D$, with some constant $\kappa\le \kappa_0$ and exponent $\beta$. Consider a set $\tilde{D} \Subset D$, and let $r_0= \frac 14 \min \left\{2,  \dist(\tilde{D}, \partial D) \right\}$.
Define $D_{r_0}:=\{x \in D: \dist(x,\partial D)\ge r_0\}$. We know that $\mathbf{u}\in C^{0,\eta}(D_{r_0}; \mathbb{R}^m)$ by Theorem \ref{T3}.
Let 
$$M:= \|\mathbf{u}\|_{L^\infty(D_{r_0}; \mathbb{R}^m)}.$$

For an arbitrary point $x_0\in \tilde{D} \cap \{|\mathbf{u}| > 0\}$, in order to estimate $|\nabla \mathbf{u}(x_0)|$, we distinguish two cases:

\medskip

\noindent {\bf Case I:} {($d:=\dist(x_0, F(\mathbf{u}))\le r_0$)}
\medskip

\noindent
Choose $y_0\in \partial B_d(x_0)\cap F(\mathbf{u})$. Then, according to Proposition \ref{SSSP-01}, for any $x\in B_{d}(x_0)$, we have
$$ |\mathbf{u}(x)| \le CM|x-y_0| \le 2CMd. $$
Note that $B_{2d}(y_0) \subset D_{r_0}$, and so $|\mathbf{u}| \le M$ in $B_{2d}(y_0)$.
The scaling function $\mathbf{u}_d(x):=\frac{\mathbf{u}(x_0+dx)}{d}$ is an almost-minimizer of the functional
\begin{equation}\label{Lip:eq2}
 \mathbf{v} \mapsto \int_{B_1} \sum_{i=1}^{m} |\nabla v_i|^p + \lambda \, dx, 
 \end{equation}
with the constant $\kappa d^{\beta}$ and exponent $\beta$, in $B_1$. Moreover, $|\mathbf{u}_d|\le 2CM$.
By Corollary \ref{imp-cor}, we obtain that
$$
|\nabla \mathbf{u}(x_0)|=|\nabla \mathbf{u}_d(\mathbf{0})| \le \tilde C,
$$
where $\tilde C$ depends only on $p, m, n, \kappa_0r_0^{\beta}, \beta, \lambda^{\frac{1}{p}}$, and $CM$.

\medskip

\noindent {\bf Case II:} {($d \ge r_0$)}
\medskip

\noindent
Define $\mathbf{u}_{r_0}(x):=\frac{\mathbf{u}(x_0+r_0x)}{r_0}$. Then 
 $\mathbf{u}_{r_0}$ is an almost-minimizer of \eqref{Lip:eq2} with the constant $\kappa r_0^{\beta}$ and exponent $\beta$, in $B_1$, which also satisfies $ \|\mathbf{u}_{r_0}\|_{L^\infty(B_1)} \le \frac{M}{r_0}$.
Thus, $|\nabla \mathbf{u}(x_0)|=|\nabla \mathbf{u}_{r_0}(\mathbf{0})| \le \bar C$, where $\bar C$ depends only on $p, m, n, \kappa_0 r_0^{\beta}, \beta, \lambda^{\frac{1}{p}}$, and $\frac{M}{r_0}$.
\end{proof}

\medskip


\section{Boundary Lipschitz regularity ($p\geq 2$)}
\label{LLL-section}

 In this section, we delve into the boundary behavior of almost-minimizers. To begin, we recall that for a $C^{1,\alpha}$-smooth domain $D$, as mentioned in Remark \ref{Remark-e-02}, the Lipschitz continuity of almost-minimizers holds up-to the fixed boundary points on $\partial D$ except for those located at $\overline{F(\mathbf{u})} \cap \partial D$. These particular points are referred to as contact points.

It might be natural to assume that the appropriate boundary versions of Propositions \ref{almost-minimizer-caccioppoli} and \ref{LConverg} should extend to cover the result of Proposition \ref{SSSP-01}, specifically to achieve linear growth around the contact points.

However, regrettably, the proof of Proposition \ref{SSSP-01} does not extend to the (general) contact points. More precisely, upon examining the proof, it becomes evident that if:
$$ \liminf_{r \to 0} \frac{\left|B_r(x_0) \cap D \cap \{|\mathbf{u}|=0\}  \right|}{|B_r(x_0) \cap D|} > 0,$$
then Lipschitz regularity up-to $x_0$ can be achieved once again. However, in general, the argument concludes with a non-trivial $p$-harmonic function in a half-space with a zero value at $x_0$. This situation does not contradict the strong minimum principle, which is essential for the indirect argument of Proposition \ref{SSSP-01}, to re-establish linear growth.

The objective of this section is to bridge this gap, assuming also that $p \geq 2$. Before delving into it, we gather the following estimates.

In the sequel, we   assume $\partial D$ to be $C^{1,\alpha}$, for some $\alpha >0$ and we will  use the following notations
$$ B_r^+ (z) :=B_r(z) \cap D, 
\qquad \quad  \omega_r(z,\mathbf{u}):=\frac{\left| B^+_r (z) \cap \left\{ |\mathbf{u}|=0 \right\}   \right|}{\left| B_r^+(z) \right|}, $$ 
for any $z \in \overline{F(\mathbf{u})}$.

The approach outlined below also offers an alternative proof for Theorem \ref{T1} when $p \ge 2$. 
We initially prove  the extent to which almost-minimizers deviate from their $p$-harmonic replacements.

\begin{lemma} 
\label{S-4-B-01}
Assume that $D$ is a $C^1$-smooth domain, and $2\le p<\infty$. Let $\mathbf{u}=(u_1, \cdots, u_m)$ be an almost-minimizer of $J$ in $D$, with some constant $\kappa$ and exponent $\beta$, and the prescribed Lipschitz boundary value $\mathbf{g}\in C^{0,1}(D;\mathbb{R}^m))$. 
Define $v_i$ to be the $p$-harmonic replacement of ${u_i}$ in $B_r^+(z)$, $z \in \overline{F(\mathbf{u})}$, and let $\mathbf{v}=(v_1, \cdots, v_m)$. Then,
\begin{equation}
\label{S-4-B-02}
\left\| \mathbf{u}  - \mathbf{v} \right\|_{L^{\infty}(B_{\frac{r}{2}}^+(z); \mathbb{R}^m )} \leq C \left( r^{1+\frac{\beta}{2(n+p)}} + r^{1-\epsilon} \omega_r(z,\mathbf{u})^{\frac{1}{2(n+p)}} \right),
\end{equation}
for any $\epsilon \in (0,1)$, and the constant $C=c(p, n, \kappa, \beta, \epsilon ,\lambda, \| \nabla \mathbf{u} \|_{L^p(D; \mathbb{R}^m)}, \| \mathbf{g} \|_{C^{0,1}(D;\mathbb{R}^m)})$.
\end{lemma}

\begin{proof}
By using $\mathbf{v}$ in the definition of almost-minimizer, $\mathbf{u}$, and invoking the $p$-energy minimality of $v_i$'s, we have
$$ \int_{B_r^+(z)} \sum_{i=1}^{m} \left(|\nabla u_i|^p - |\nabla v_i|^p \right) \, dx \leq \kappa r^{\beta} \int_{B_r^+(z)} \sum_{i=1}^{m} |\nabla u_i|^p \, dx + \kappa \lambda r^{\beta} |B_r^+(z)| + \lambda \left| B_r^+(z) \cap \left\{ |\mathbf{u}|=0 \right\}  \right|. $$
Now, using the uniform H\"older continuity of $u_i$'s (see \eqref{Fi-14} in Theorem \ref{T3} and also remark \ref{Remark-e-01}), we get
\begin{equation}
\label{S-4-B-01.5}
\int_{B_r^+(z)} \sum_{i=1}^{m} \left(|\nabla u_i|^p - |\nabla v_i|^p \right) \, dx \leq L r^{n+\beta-\epsilon}  + \lambda \omega_r(z, \mathbf{u})  \left| B_r^+(z) \right|,
\end{equation}
for any $\epsilon>0$, where
\begin{equation}
\label{LLL-0e-00}
L=c(p, n, \kappa_0, \epsilon, \|\mathbf{g} \|_{C^{0,1}(D;\mathbb{R}^m)}) \left( \| \nabla \mathbf{u} \|^p_{L^p(D; \mathbb{R}^m)}+\lambda \right).
\end{equation} 
Now, recalling the following inequality from \cite[equation (3.13)]{de2023nonuniformly}
$$ c_p |z_1-z_2|^p \leq \left| \mathbf{V}(z_1)- \mathbf{V}(z_2) \right|^2,
$$
where $\mathbf{V}$ is again the map \eqref{auxillary-map}, 
we get
\begin{align}\label{G-inequality}
c_p \int_{B_r^+(z)} & |\nabla u_i - \nabla v_i|^p \, dx
\leq \int_{B_r^+(z)} \left| \mathbf{V}(\nabla u_i)-\mathbf{V}(\nabla v_i) \right|^2 \, dx. 
\end{align}
At this point, we remind that by using \eqref{Fi-4}, again we have 
\begin{equation}
\label{LLL-0e-1}
c \int_{B_r^+(z)} \left| \mathbf{V}(\nabla u_i)-\mathbf{V}(\nabla v_i) \right|^2 \, dx \leq \int_{B_r^+(z)} \left(\left| \nabla u_i \right|^p-\left| \nabla v_i \right|^p \right) \, dx.
\end{equation}
Now, by considering \eqref{G-inequality} and \eqref{LLL-0e-1}, then \eqref{S-4-B-01.5} implies  
\begin{equation}
\label{S-4-B-01.25}
\int_{B_{r}^+(z)} | \nabla ({u_i} - {v_i})|^p \, dx \leq c \left( L  r^{n+\beta-\epsilon} + \lambda r^n \omega_r(z, \mathbf{u}) \right),
\end{equation}
with $c=c(p)$. Since $u_i-v_i \in W_0^{1,p}(B_r^+(z))$, by Poincar\'e inequality we conclude that
\begin{equation}
\label{S-4-B-03}
\int_{B_r^+(z)} | {u_i} - {v_i} |^{p} \, dx \leq c  r^p \left( L  r^{n+\beta-\epsilon} + \lambda r^n \omega_r(z, \mathbf{u}) \right).
\end{equation}
Since both $u_i$ and $v_i$ are uniformly H\"older continuous in $B_{\frac{3r}{4}}^+(z)$, for any exponent $\alpha \in (0,1)$, 
then if $ \mu=|(u_i-v_i)(y_0)|$, say at $y_0 \in B_{\frac{r}{2}}^+(z)$,
we get
\begin{equation}
\label{LLL-0e-1.5}
|(u_i-v_i)(x)| \geq |(u_i-v_i)(y_0)| - h_0|x-y_0|^{\alpha} \geq \mu - h_0 r_0^{\alpha}, \qquad \text{in} \quad B_{r_0}(y_0),
\end{equation}
where $r_0:=\min \left\{ \frac r4, \left(\frac\mu {2h_0} \right)^{\frac1\alpha} \right\}$  and 
$h_0$ is an upper-bound for the H\"older norm of $u_i-v_i$ which can be considered a universal constant, 
\begin{align} \nonumber
   \| u_i-v_i \|_{C^{0,\alpha}(B_{r_0}(y_0))} 
& \leq \|\mathbf{u}\|_{C^{0,\alpha}(B_{r_0}(y_0); \mathbb{R}^m)} + \| \mathbf{v} \|_{C^{0,\alpha}(B_{r_0}(y_0); \mathbb{R}^m)} \\ \nonumber
& \leq c(p, n, \kappa, \beta, \alpha, \| \mathbf{g} \|_{C^{0,1}(D;\mathbb{R}^n)}) \left( \| \nabla \mathbf{u} \|_{L^p\left(D; \mathbb{R}^m\right)} + \lambda^{\frac{1}{p}} \right) + c(p, n, \| \mathbf{g} \|_{C^{0,1}(D;\mathbb{R}^n)}) r^{-\alpha} \| \mathbf{u} \|_{L^\infty\left(B_r^+(z); \mathbb{R}^m\right)} \\ \label{LLL-0e-2}
& \leq C(p, n, \kappa, \beta, \alpha, \| \mathbf{g} \|_{C^{0,1}(D;\mathbb{R}^n)}) \left( \| \nabla \mathbf{u} \|_{L^p\left(D; \mathbb{R}^m\right)} + \lambda^{\frac{1}{p}} \right)=:h_0.
\end{align}
Since $2 h_0 r_0^{\alpha} \le \mu $, then
$$ |u_i-v_i| \geq \frac{\mu}{2}, \qquad \text{in} \quad B_{r_0}(y_0). $$
From this and \eqref{S-4-B-03}, we get
$$ |B_1| r_0^n \frac{\mu^p}{2^p} \leq c \left( L  r^{n+p+\beta-\epsilon} + \lambda r^{n+p} \omega_r(z,\mathbf{u}) \right). $$ 
If $r_0= \left(\frac\mu {2h_0} \right)^{\frac1\alpha} $, by substituting the value of $r_0$,
$$ \mu^{\frac{n}{\alpha}+p} \leq c h_0^{\frac{n}{\alpha}} \left( L  r^{n+p+\beta-\epsilon} + \lambda r^{n+p} \omega_r(z,\mathbf{u}) \right). $$
Hence, $\mu$ should satisfy
$$ \mu \leq c h_0^{\frac{n}{n+p\alpha}} \left( L  r^{n+p+\beta-\epsilon} + \lambda r^{n+p} \omega_r(z,\mathbf{u}) \right)^{\frac{\alpha}{n+p\alpha}}.$$
Choosing  $\alpha < 1$ close enough to 1, and by substituting \eqref{LLL-0e-00} and \eqref{LLL-0e-2} in the above inequality, finally, we arrive at
$$
\begin{aligned}
\mu 
& \leq C \left(  r^{1+\frac{\beta}{2(n+p)}} + r^{1-\epsilon} \omega_r(z,\mathbf{u})^{\frac{1}{2(n+p)}} \right).
\end{aligned}
$$
where $C=c(p, n, \kappa, \beta, \epsilon,  \lambda, \| \nabla \mathbf{u} \|_{L^p(D; \mathbb{R}^m)}, \| \mathbf{g} \|_{C^{0,1}(D;\mathbb{R}^n)})$.
If $r_0=\frac r4$, then similarly from \eqref{S-4-B-03} we get 
$$ |B_1| \frac{r^n}{4^n} \frac{\mu^p}{2^p} \leq c \left( L  r^{n+p+\beta-\epsilon} + \lambda r^{n+p} \omega_r(z,\mathbf{u}) \right), $$ 
and then
$$
\mu \le C\left(r^{1+\frac\beta 2} + r  \omega_r(z,\mathbf{u})\right).
$$
This completes the proof of \eqref{S-4-B-02}.
\end{proof}

We also prove a generalization of \cite[Lemma 10.1]{MR3385167} to any $2\le p<\infty$.

\begin{proposition}
\label{Lemma10.1}
Assume that $D$ is a 
$C^1$-smooth domain, $2\le p<\infty$, and  $\mathbf{u}=(u_1, \cdots, u_m)$ is  an almost-minimizer of $J$ in $D$, with constant $\kappa$ and exponent $\beta$. Also, let $v_i$ as the $p$-harmonic replacement of $u_i$ in $B_{r}^+(z)$, $z \in \overline{F(\mathbf{u})}$. Then,
\begin{equation}
\label{Eq-Lemma10.1}
\int_{B_{r}^+(z)} \sum_{i=1}^{m} \left| \left(u_i-v_i \right)^+ \right|^p \, dx \leq C \kappa r^{p+\beta} \left( r^n + \int_{B_r^+(z)} \sum_{i=1}^{m} |\nabla u_i|^p \, dx \right),
\end{equation}
where  $f^+:=\max \left\{f,0 \right\}$ and 
$C= C(p, n)$.
\end{proposition}

\begin{proof}
We define the following auxiliary functions
\begin{equation}
\label{LLL-e-1}
w_i(x)=
\begin{cases}
u_i(x), \qquad & \text{for} \quad x \in D \setminus B_{r}^+(z), \\
\min\{ u_i(x), v_i(x) \}, \qquad & \text{for} \quad x \in B_{r}^+(z),
\end{cases}
\end{equation}
which has the properties
$w_i \in W^{1,p}\left(B_{r}^+(z)\right)$ and  $w_i=u_i$ on $\partial B_{r}^+(z)$. 
By using $\mathbf{w}=(w_1, \cdots, w_m)$ in the definition of almost-minimizer, $\mathbf{u}$, we have
\begin{equation}
\label{LLL-e-2}
J \left(\mathbf{u}; B_{r}^+(z) \right) \leq \left( 1+ \kappa r^{\beta} \right) J \left(\mathbf{w}; B_{r}^+(z) \right).
\end{equation}
We further define the sets $G_i :=\{ x \in B_{r}^+(z) \, : \, w_i(x) \neq u_i(x) \} = \{ x \in B_{r}^+(z) \, : \, v_i(x) < u_i(x)\}$. Since $\nabla w_i = \nabla u_i $ a.e. in $B_{r}^+(z) \setminus G_i$, and $\nabla w_i = \nabla v_i$ a.e. in $G_i$, by \eqref{LLL-e-2}, we get
\begin{align} 
J \left(\mathbf{w}; B_{r}^+(z) \right) - J \left(\mathbf{u}; B_{r}^+(z) \right) 
\leq \int_{B_{r}^+} \sum_{i=1}^{m} \left(|\nabla w_i|^p - |\nabla u_i|^p \right) \, dx 
\label{LLL-e-3}
\leq  \sum_{i=1}^{m} \int_{G_i} \left( |\nabla v_i|^p - |\nabla u_i|^p \right) \, dx. 
\end{align}
Next we claim that 
\begin{equation} 
\label{LLL-e-4} 
 \sum_{i=1}^{m} \int_{G_i} \left| \mathbf{V}(\nabla u_i) - \mathbf{V}(\nabla v_i) \right|^2 \, dx \leq c \sum_{i=1}^{m}  \int_{G_i} \left( \left| \nabla u_i \right|^p - \left| \nabla v_i \right|^p \right) \, dx.
\end{equation}
To establish \eqref{LLL-e-4}, begin by defining $z_i(x):=\max\left\{u_i(x), v_i(x) \right\}$ for $x \in B_{r}^+(z)$. Notably, this function belongs to $W^{1,p}(B_{r}^+(z))$, and its trace coincides with $v_i$ and $u_i$ on $\partial B_{r}^+(z)$.

Now, leveraging the $p$-harmonicity of $v_i$, it is straightforward to observe that:
\begin{equation*} 
\int_{B_{r}^+(z)} \left| \nabla v_i \right|^p   \,dx = \int_{B_{r}^+(z)} \left|\nabla v_i \right|^{p-2} \left( \nabla v_i \cdot \nabla z_i \right)  \, dx. 
\end{equation*}
Hence, using \eqref{Fi-4} gives \eqref{LLL-e-4}, and the claim is proved. 
Now, \eqref{LLL-e-4}, \eqref{LLL-e-3} and \eqref{LLL-e-2}, together with \eqref{G-inequality} yield 
$$
\begin{aligned}
\sum_{i=1}^{m} \int_{G_i} \left| \nabla u_i - \nabla v_i \right|^p \, dx \leq c \kappa r^{\beta} J \left(\mathbf{w}; B_{r}^+(z) \right) 
& \leq C \kappa r^{\beta} \left( r^n + \int_{B_{r}^+(z)} \sum_{i=1}^{m} |\nabla w_i|^p \, dx \right) \\ &  
\leq C \kappa r^{\beta} \left( r^n + \int_{B_{r}^+(z)} \sum_{i=1}^{m} |\nabla u_i|^p \, dx \right),
\end{aligned}
$$
where in the last line, we have used the definition of $w_i$, and the $p$-energy minimality of $v_i$. 
Finally, applying the Poincar\'e inequality to the function $ (u_i-v_i)^+ $, gives
\begin{align*} \nonumber
\int_{B_{r}^+(z)} \sum_{i=1}^{m} \left| \left(u_i-v_i \right)^+ \right|^p \, dx & \leq C r^p \int_{B_{r}^+(z)} \sum_{i=1}^{m}  \left| \nabla \left(u_i-v_i \right)^+ \right|^p \, dx \leq C \kappa r^{p+\beta} \left( r^n + \int_{B_r^+(z)} \sum_{i=1}^{m} |\nabla u_i|^p \, dx \right),
\end{align*}
since $u_i \leq v_i $ in $B_{r}^+(z) \setminus G_i$, and hence $\nabla \left( (u_i-v_i)^+ \right) = \chi_{G_i} \left( \nabla u_i - \nabla v_i \right)$. 
This completes the proof.
\end{proof}

Proposition \ref{Lemma10.1} has the following immediate consequence.

\begin{corollary}
\label{LLL-c-1}
Assume that $D$ is a 
$C^1$-smooth domain, and $2\le p<\infty$. Let $\mathbf{u}=(u_1, \cdots, u_m)$ be an almost-minimizer of $J$ in $D$, with some constant $\kappa$ and exponent $\beta$, and the prescribed Lipschitz boundary value $\mathbf{g} \in C^{0,1}(D;\mathbb{R}^m)) $.
Define $v_i$ to be the $p$-harmonic replacement of ${u_i}$ in $B_r^+(z)$, $z \in \overline{F(\mathbf{u})}$, and let $\mathbf{v}=(v_1, \cdots, v_m)$. Then,
\begin{equation}
\label{S-4-B-02*}
\left\| \left(u_i  - v_i \right)^+ \right\|_{L^{\infty} (B_{\frac r2}^+(z) )} \leq C  r^{1+\frac{\beta-\epsilon}{n+p}},
\end{equation}
for any $\epsilon>0$, and the constant $C=c(p, n, \kappa, \beta, \epsilon, \lambda,  \| \nabla \mathbf{u} \|_{L^p(D; \mathbb{R}^m)}, \|\mathbf{g}\|_{C^{0,1}(D; \mathbb{R}^m)})$.
\end{corollary}

\begin{proof}
If $ \mu=(u_i-v_i)(y_0)$, say at $y_0 \in G_i=\{ x \in B_{r}^+(z) \, : \, v_i(x) < u_i(x)\}$, then, by the uniform H\"older continuity of $u_i-v_i$, for any exponent $\alpha \in (0,1)$, and the same arguments of \eqref{LLL-0e-1.5} and \eqref{LLL-0e-2}, we get
$$ 
(u_i-v_i)(x) \geq (u_i-v_i)(y_0) - h_0|x-y_0|^{\alpha} \geq \mu - h_0 r_0^{\alpha}, \qquad \text{in} \quad B_{r_0}(y_0) \cap G_i,
$$
where $h_0 = C(p, n, \kappa, \beta, \alpha,, \|\mathbf{g}\|_{C^{0,1}(D; \mathbb{R}^m)}) \left( \| \nabla \mathbf{u} \|_{L^p\left(D; \mathbb{R}^m\right)} + \lambda^{\frac{1}{p}} \right)$ is the upper-bound for the H\"older norm of $u_i-v_i$, and
also $r_0:=\min \left\{ \frac r4, \left(\frac\mu {2h_0} \right)^{\frac1\alpha} \right\}$.
Then, 
$$ (u_i-v_i)(x) \geq \frac{\mu}{2}, \qquad \text{in} \quad B_{r_0}(y_0), $$
and therefore $B_{r_0}(y_0) \subset G_i$.
From this and \eqref{Eq-Lemma10.1}, together with the boundary version of \eqref{Fi-14} (Remark \ref{Remark-e-01}), we get 
$$ |B_1| r_0^n \frac{\mu^p}{2^p} \leq c \kappa r^{p+\beta} \left( r^n+ c  r^{n-\epsilon} \right), $$ 
for any $\alpha \in (0,1)$, and any $\epsilon >0$.
If $r_0= \left(\frac\mu {2h_0} \right)^{\frac1\alpha}$, by substituting the value of $r_0$, and choosing an appropriate value of $\alpha< 1$ (close enough to 1), we get
$$
\mu \leq C  r^{1+\frac{\beta-\epsilon}{n+p}},
$$
for any $\epsilon >0$, and the constant $C=c(p, n, \kappa, \beta, \epsilon, \lambda,  \| \nabla \mathbf{u} \|_{L^p(D; \mathbb{R}^m)}, \|\mathbf{g}\|_{C^{0,1}(D; \mathbb{R}^m)})$. 
Also, for the case $r_0= \frac r4$, we will get
\[
\mu \le C r^{1+\frac{\beta-\epsilon} p}.
\]
This completes the proof.
\end{proof}

Now, we present the following general approximation result, the proof of which involves a minor modification of 
\cite[Lemma 2.5]{fotouhi2023weakly}.

\begin{proposition}
\label{Lemma2.5}
Assume that $D$ is a 
$C^1$-smooth domain, $x_0\in D$, and $1<p<\infty$. Let $u \in W^{1,p}(D)$ be a non-negative function. Then, there exists $c>0$, depending only on $p$, $n$, and the regularity of the domain $D$, such that for any $r>0$ we have
\begin{equation}
\label{LLL-e-0}
\left( \frac{1}{r} \sup_{B_{\frac{r}{2}}(x_0) \cap D} v \right)^p \left| B_r(x_0) \cap \{ u =0 \}  \right| \leq c \int_{B_r(x_0) \cap D} | \nabla (u-v)|^p \, dx,
\end{equation}
where $v$ is the $p$-harmonic replacement of $u$ in $B_r(x_0) \cap D$.
\end{proposition}

As stated earlier in this section, our objective is to establish the Lipschitz regularity of almost-minimizers at the contact points. It's worth noting that Lipschitz regularity (and even $C^{1,\tilde{\eta}}$-regularity) near Dirichlet data and in the absence of free boundary points  has been previously discussed (refer to Remark \ref{Remark-e-02}). Our primary focus is on the following result.

\begin{theorem}
\label{Bdr-result-contact}
Assume that $D$ is a 
$C^1$-smooth domain, and $2\le p<\infty$. Let $\mathbf{u}=(u_1, \cdots, u_m)$ be an almost-minimizer of $J$ in $D$, with constant $\kappa$ and exponent $\beta$, the prescribed Lipschitz boundary value $\mathbf{g}\in C^{0,1}(D;\mathbb{R}^m))$, and moreover $z \in \overline{F(\mathbf{u})}$. 
Then, $\mathbf{u}$ has linear growth at $z$; in other words there exist constants $r_0$ and $C$ depending only on $p,n,\kappa, \lambda ,\| \nabla \mathbf{u} \|_{L^p(D; \mathbb{R}^m)} $, and $\|\mathbf{g}\|_{C^{0,1}(D;\mathbb{R}^m)}$ such that
\[
\sup_{B_r^+(z)}|\mathbf{u} | \le Cr, \qquad \text{ for all } r\le r_0.
\]
\end{theorem}

\begin{proof}

Invoking \eqref{LLL-e-0} with $v=v_i$ and $u=u_i$, and applying the estimate \eqref{S-4-B-01.25}, yields
\begin{equation}
\label{LLL-e-8.5}
\left( \frac{1}{r} \sup_{B_{\frac{r}{2}}^+(z)} v_i \right)^p \left| B_r^+(z) \cap \{ u_i =0 \} \right| \leq C \left(  r^{n+\beta-\epsilon} + \lambda r^n\omega_r(z,\mathbf{u}) \right),
\end{equation}
 for any $\epsilon >0 $; here  $C=c(p, n, \kappa, \epsilon,\|\mathbf{g}\|_{C^{0,1}(D;\mathbb{R}^m)}) \left( \| \nabla \mathbf{u} \|^p_{L^p(D; \mathbb{R}^m)}+\lambda \right)$.
We claim next that 
\begin{equation}
\label{LLL-e-9.5}
\frac{1}{r} \sup_{B_{\frac{r}{2}}^+(z)} v_i \leq  (2C\lambda)^{\frac1p}, \qquad \text{for any $0<r<r_0$, and $r_0$ small enough}.
\end{equation}
To prove this, by an indirect argument, assume that 
\begin{equation}
\label{LLL-e-10}
\frac{1}{r} \sup_{B_{\frac{r}{2}}^+(z)} v_i > (2C\lambda)^{\frac1p}.
\end{equation}
Then, \eqref{LLL-e-8.5} implies
\begin{equation}
\label{LLL-e-11}
C\lambda \omega_{r}(z,\mathbf{u}) \leq C r^{\beta-\epsilon}.
\end{equation}
Putting \eqref{LLL-e-11} in \eqref{S-4-B-02}, and selecting appropriate values of $\epsilon$ in \eqref{S-4-B-02} and \eqref{LLL-e-11}, implies
\begin{equation}
\label{LLL-e-12}
\left\| u_i  - v_i \right\|_{L^{\infty}(B_{\frac{r}{2}}^+(z))} \leq C \left(  r^{1+\frac{\beta}{2(n+p)}} +  r^{1+\frac{\beta}{4(n+p)}} \right) \leq C r^{1+\frac{\beta}{4(n+p)}}.
\end{equation}
Observing that $z \in F(\mathbf{u})$ and considering \eqref{LLL-e-12}, we obtain, along with the Harnack inequality, 
$$ \sup_{B_{\frac{r}{2}}^+(z)} v_i \leq C \inf_{B_{\frac{r}{2}}^+(z)} v_i \leq C r^{1+\frac{\beta}{4(n+p)}}, $$
which, for sufficiently small $r_0$,  
contradicts \eqref{LLL-e-10} and establishes  claim \eqref{LLL-e-9.5}.
This in turn implies the linear growth of $v_i$ at $z$, i.e.
\begin{equation}
\label{LLL-LG}
0 \leq v_i(x) \leq 2(2C\lambda)^{\frac1p} |x-z|, \qquad \text{for} \quad x \in B_{\frac{r_0}2(z)}^+.
\end{equation}

Now, we are ready to complete the proof. 
In the set $\{u_i \leq v_i\}$, the linear growth of $u_i$ is the direct consequence of \eqref{LLL-LG}. So, assume that we are in the set $\{u_i-v_i>0\}$. Then, by using the estimate \eqref{S-4-B-02*}, and choosing $\epsilon<\beta$, we get 
\begin{equation} 
\label{LLL-e-13} 
\left\| \left(u_i - v_i\right)^+ \right\|_{L^{\infty} \left(B_r^+(z)\right)} \leq C r^{1+\frac{\beta-\epsilon}{n+p}} \leq C r, 
\end{equation} 
or, equivalently
$$ 0 \leq u_i(x) \leq C r + \left\| v_i \right\|_{L^{\infty} \left(B_{r}^+(z) \right)}, $$
for any $x \in \{u_i-v_i>0\}$. Finally, recall \eqref{LLL-LG} to get
 the linear growth of $u_i$ at $z$.
\end{proof}

Now, we are ready to state and prove the following theorem regarding Lipschitz continuity (assuming 
$p \geq 2)$, irrespective of whether we are at contact points or free boundary points. Before that, let's emphasize the following remark.

\begin{remark}
\label{LLL-rem}
Note that the results presented in this section, namely Lemma \ref{S-4-B-01}, Proposition \ref{Lemma10.1}, Corollary \ref{LLL-c-1}, and Theorem \ref{Bdr-result-contact}, have been established for all points $z \in \overline{F(\mathbf{u})}$ in a uniform manner, rather than exclusively for the contact points $\overline{F(\mathbf{u})} \cap \partial D$. It is important to highlight that these points were the only gap in the previous sections concerning up-to-boundary Lipschitz regularity. Consequently, the following argument also provides a second proof for the local Lipschitz continuity of almost-minimizers when $p \geq 2$.
\end{remark}

\begin{theorem}[Boundary Lipschitz regularity of almost-minimizers]
\label{T1-up-to}
Assume that $D$ is a $C^{1,\alpha}$-smooth domain, and $ 2\leq p<\infty$. Let $\mathbf{u} :D \to \mathbb{R}^m $ be an almost-minimizer of $J$ in $D$ with the prescribed boundary value $\mathbf{g} \in C^{1,\alpha}(D;\mathbb{R}^m)$. Then, $\mathbf{u}$ is up-to-boundary Lipschitz continuous. 
\end{theorem}

\begin{proof}
Let $\mathbf{u}$ be an almost-minimizer of $J$ in $D$, with some constant $\kappa\le \kappa_0$ and exponent $\beta$.
For an arbitrary point $x_0 \in \{ |\mathbf{u}|>0 \}$, in order to estimate $|\nabla \mathbf{u}(x_0)|$, we argue as follows.

\medskip

Let   $r_0$ be given by Theorem \ref{Bdr-result-contact}),
and define 
$$d:=\dist \left(x_0, \overline{F(\mathbf{u})}\right),$$
and consider the following two cases:

\medskip

\noindent {\bf Case I:} ($d\le \frac{1}{2}r_0$)

\noindent
Choose $y_0 \in \partial B_d(x_0)\cap \overline{F(\mathbf{u})}$. Then, according to Theorem \ref{Bdr-result-contact} (see also Remark \ref{LLL-rem}), we have
$$ |\mathbf{u}(x)| \leq C |x-y_0| \leq Cd, $$
for any $x\in B_{d}^+(x_0)$.
Now, the scaled function $\mathbf{u}_d(x):=\frac{\mathbf{u}(x_0+dx)}{d}$, satisfying $|\mathbf{u}_d|\le C$, will be an almost-minimizer of the functional
$$ \mathbf{v} \mapsto \int \sum_{i=1}^{m} |\nabla v_i|^p + \lambda \, dx, $$
over $B_1 \cap \frac1d (D-x_0)$
with the constant $\kappa d^{\beta}$ and exponent $\beta$. 
By virtue of a modified boundary version of Corollary \ref{imp-cor}, which can be easily established (see also Remark \ref{Remark-e-02}), we obtain that 
$$
|\nabla \mathbf{u}(x_0)|=|\nabla \mathbf{u}_d(\mathbf{0})| \le \tilde C,
$$
where $\tilde C$ depends only on $p$, $m$, $n$, $\kappa_0 r_0^{\beta}$, $\beta$, $\lambda$, $\| \nabla \mathbf{u} \|_{L^p(D; \mathbb{R}^m)}$, $\|\mathbf{g}\|_{C^{0,1}(D;\mathbf{R}^m)}$, and the regularity of the domain $D$.

\medskip

\noindent {\bf Case II:} {($d > \frac{1}{2}r_0$)}

\noindent
In this case, $|\nabla \mathbf{u}(x_0)|$ is uniformly bounded by considering the Remark \ref{Remark-e-02} and arguing similar to Case II in the proof of Theorem \ref{T1}.
\end{proof}


\medskip

\paragraph{\bf{Acknowledgements}}
M. Bayrami and M. Fotouhi were supported by Iran National Science Foundation (INSF) under project No. 4001885. M. Bayrami was supported by a grant from IPM. H. Shahgholian was supported by Swedish Research Council.

\section*{Declarations}

\noindent {\bf  Data availability statement:} All data needed are contained in the manuscript.

\medskip
\noindent {\bf  Funding and/or Conflicts of interests/Competing interests:} The authors declare that there are no financial, competing or conflict of interests.


\bibliographystyle{acm}
\bibliography{mybibfile}

\end{document}